\documentclass[final,nomarks]{dmtcs-episciences}

\usepackage[utf8]{inputenc}
\usepackage{subfigure,amsmath,color}
\usepackage[round]{natbib}

\newtheorem{theorem}{Theorem} 
\newtheorem{lemma}[theorem]{Lemma}

\newcommand{\bE}{\mathbb{E}}
\newcommand{\bF}{\mathbb{F}}
\newcommand{\bB}{\mathbb{B}}

\newcommand{\bN}{\mathbb{N}}

\newcommand{\bR}{\mathbb{R}}

\newcommand{\bG}{\mathbb{G}}
\newcommand{\bS}{\mathbb{S}}
\newcommand{\bV}{\mathbb{\hspace{-.07mm}V\hspace{-.2mm}}}

\newcommand{\cF}{\mathcal{F}}
\newcommand{\cG}{\mathcal{G}}

\newcommand{\cL}{\mathcal{L}}

\newcommand{\cT}{\mathcal{T}}

\newcommand{\unif}{\text{\rm unif}}
\newcommand{\var}{\text{\rm var}}

\newcommand{\todistr}{\to_{\text{\rm\tiny distr}}}

\newcommand{\E}{{\bE}}

\title[Randomly growing discrete structures]{Persisting randomness in randomly growing  discrete structures: graphs and
  search trees}

\author{Rudolf Gr\"ubel}
\affiliation{
         Leibniz Universit\"at Hannover, Germany} 

\allowdisplaybreaks
\received{2015-05-06}
\revised{2015-09-07}
\accepted{2015-09-09}
\begin{document}
\publicationdetails{18}{2015}{1}{1}{644}
\maketitle

\keywords{Boundary theory, Markov chains, random graphs, search trees.}

\maketitle

\begin{abstract} 
  The successive discrete structures generated by a sequential algorithm from random input constitute a Markov chain
  that may exhibit long term dependence on its first few input values. Using examples from random graph theory and
  search algorithms we show how such persistence of randomness can be detected and quantified with techniques from
  discrete potential theory. We also show that this approach can be used to obtain strong limit theorems in cases where 
  previously only distributional convergence was known.
\end{abstract}

\section{Introduction}\label{sec:intro}
Given a sequence $t_1,t_2,\ldots$ of input values, a sequential algorithm produces an output sequence $x_1,x_2,\ldots$,
where the next output $x_{n+1}$ depends on the current state $x_n$ and the next input $t_{n+1}$ only. For cases where
$x_n$ is a discrete structure, such as a permutation or a graph, and where the input values are realizations of
independent random variables with the same distribution, the output sequence is a Markov chain $X=(X_n)_{n\in\bN}$ that
is adapted to a combinatorial family $\bF$ in the sense that $X_n$ takes its values in the subset $\bF_n\subset \bF$ of
objects with base parameter $n$. Markov chains of this type often exhibit \emph{persisting randomness}: Informally, this
means that the influence of early values does not disappear as time goes by; formally, it means that the tail
$\sigma$-field $\cT(X)$ associated with $X$ is not trivial. Further, such chains eventually leave every fixed finite
subset of $\bF$ with probability~1, which leads to the related problem of finding a state space completion that captures
the information contained in $\cT(X)$. 

The classical example is the P\'olya urn: Initially, at time $n=1$, the urn contains one red and one blue ball. At time
$n=2,3,\ldots$, a ball is selected uniformly at random and put back, together with another ball of the same colour. Here
$\bF_n$ is the set of the pairs $(i,j)\in\bN_0\times\bN_0$ with $i+j+1=n$, where $i$ and $j$ are the number of balls of the
two colours added up to time $n\in\bN$. A suitable augmentation of the state space $\bF=\bN_0\times\bN_0$ is obtained by
regarding a sequence $((i_n,j_n))_{n\in\bN}$ with $i_n+j_n\to\infty$ as convergent if and only if $(i_n+1)/(i_n+j_n+2)$
(the proportion of red balls) tends to a value $\alpha\in [0,1]$ as $n\to\infty$, which leads to a state space
completion $\bar\bF$ that may be represented by $\bF\cup [0,1]$. With respect to this convergence we have $X_n\to
X_\infty$ almost surely, where $X_\infty$ generates the (non-trivial) tail $\sigma$-field~$\cT(X)$.

Discrete potential theory provides a general method for the construction of such state space boundaries. This was
initiated in a fundamental paper by~\cite{Doob59} and has been applied to P\'olya urns by~\cite{BK1964}. A recent textbook treatment
is given in~\cite{Woess2}; see also the survey by \cite{Saw97}.  Many authors have used boundary theory for the
analysis of random walks on discrete structures; see~\cite{KaimVersh} for a very influential review, and the more recent
monograph~by \cite{Woess1}. 

In the present paper we regard the discrete structures themselves as states of a stochastic
process.  Some standard search algorithms have recently been investigated from this point of view in~\cite{EGW1}, and
the results have been used in~\cite{GrMtree} to prove strong limit theorems for functionals of the output sequence,
such as the path length or Wiener index of search trees, that have attracted the attention of many researchers.  
We address the phenomenon of randomness persistence with these tools, specifically in connection with some
popular models for random graphs, and for search algorithms. Further, we show that the method can be used to obtain
almost sure convergence for the structures themselves or for functionals of the structures in cases where previously
only convergence in distribution was known.

In the next section we provide some background on discrete potential theory and the Doob-Martin compactification.  In
order to keep this short we restrict ourselves to combinatorial Markov chains, where the state space is graded by the
time parameter.  Section~\ref{sec:graph} gives some elementary examples from random graph theory; the present author is
not aware of any previous use of the Doob-Martin compactification in connection with (general) graph limits.  In
Section~\ref{sec:BST} we consider search trees, where we can build on the work of~\cite{EGW1} and~\cite{GrMtree}.  In a final
section we collect some comments on related work and provide further pointers to the literature.

We hope that such results contribute to the theoretical understanding of the growth models and algorithms. 
From an entirely practical point of view persisting randomness should be of interest as a strong dependence on the first
few input values may be an entirely unwelcome aspect of an algorithm that the practitioner may have to address, for
example by an additional randomization step.

\section{An ultrashort summary of Markov chain boundary theory}%
\label{sec:summ}

A Markov chain is a sequence $X=(X_n)_{n\in\bN}$ of random variables that take their values in some countable
set $\bF$, the state space, such that the Markov property holds,
\begin{equation}\label{eq:Markov}
  P(X_{n+1}=x_{n+1}|X_n=x_n,\ldots,X_1=x_1)\; =\;   P(X_{n+1}=x_{n+1}|X_n=x_n) 
\end{equation}
for all $n\in\bN$, $x_1,\ldots,x_{n+1}\in\bF$. In the cases we are interested in there will be a canonical state
$e\in\bF$ with $X_1=e$, and the transitions are homogeneous in time, which means that for some 
function $p:\bF\times\bF\to [0,1]$, 
\begin{equation*}
  P(X_{n+1}=y|X_n=x) = p(x,y)\quad\text{for all } x,y\in\bF, \; n\in\bN. 
\end{equation*}
These are the transition probabilities; together with the starting point $e$ they determine the distribution 
of the stochastic process $X$. We also assume that
\begin{equation}\label{eq:weakirr}
  P(X_n=x\ \text{ for some }n\in\bN) \, > \, 0 \ \text{ for all } x\in \bF.
\end{equation}
In words:  Every state has a chance to be visited---the chain is weakly irreducible.

Boundary theory provides an approach to the asymptotics of chains that `leave the state space' in the
sense that $\lim_{n\to\infty} P(X_n\in S)=0$ for every finite set $S\subset \bF$. It gives
the `right' extension (completion, compactification) $\bar \bF$ of the state space, in the sense that
\begin{equation}\label{eq:aslimit}
  X_n\to X_\infty \ \text{ as } n\to\infty\ \text{ with probability } 1
\end{equation}
for some random variable $X_\infty$ with values in the boundary $\partial\bF$ of $\bF$ in $\bar\bF$,
and that 
\begin{equation}\label{eq:tailgen}
          \sigma(X_\infty) =_{\text{\rm a.s.}} \cT(X) := \bigcap_{n=1}^\infty \sigma\bigl(\{X_m:\, m\ge n\}\bigr).
\end{equation}
In words: The limit generates the tail $\sigma$-field of the process, up to null
sets. For property~\eqref{eq:tailgen} we assume that $X$ has the space-time property, by which we mean
that each state can be visited at one particular point in time only. The combinatorial Markov chains 
in Section~\ref{sec:intro} are such space-time processes.

This feat is achieved by the Doob-Martin compactification, where we regard a sequence $(y_n)_{n\in\bN}\subset\bF$ as
convergent if the conditional probabilities $P(X_1=x_1,\ldots,X_m=x_m|X_n=y_n)$ converge as $n\to\infty$ for all fixed
$m\in\bN$, $x_1,\ldots,x_m\in \bF$. Due to the Markov property~\eqref{eq:Markov} the construction can be based on the 
Martin kernel $K$,
\begin{equation*}
  K(x,y) :=\frac{P(X_n=y| X_m=x)}{P(X_n=y)},
    \quad x,y \in \bF,\, n>m, 
\end{equation*}
where $m$ and $n$ are the time values associated with the states $x$ and $y$ respectively; here 
weak irreducibility~\eqref{eq:weakirr} is important. Indeed, manipulations of elementary 
conditional probabilities lead to
\begin{equation*}
  P(X_1=x_1,\ldots,X_m=x_m|X_n=y_n) = K(x_m,y_n)\, P(X_1=x_1,\ldots,X_m=x_m),
\end{equation*}
which connects the convergence condition on the conditional probabilities to the convergence of the values of the
Martin kernel. We mention in passing that this approach to Doob-Martin convergence is
equivalent to the usual approach via potential kernels; see also~\cite[Section~3]{EGW1} and ~\cite[Section~2]{EGW2}.

From a general point of view, any family $\cF$ of functions $f:\bF\to\bR$ that separates the points of $\bF$ leads to 
an embedding of $\bF$ into the space $\bR^\cF$ of functions from $\cF$ to $\bR$ via
\begin{equation*}
  x\, \mapsto \, \bigl( f\mapsto f(x)\bigr).
\end{equation*}
On $\bR^\cF$ we use the topology of pointwise convergence.  
If all $f\in\cF$ are continuous (which they automatically are if we endow $\bF$ with the discrete topology)
and bounded, then the embedding is continuous and its range is a product of bounded intervals, hence compact by
Tychonov's theorem. This is a variant of the Stone-{\v C}ech compactification, see \cite[p.152f]{Kel55}. In this construction, 
all $f\in\cF$ have a unique continuous extension to the whole of the compactified space. Alternatively, for a countable
family $\cF$, a suitable metric can be defined on $\bF$ using these functions, such that the associated completion 
has these properties; see~\cite[p.187]{Woess2}.

In our present setup, the Doob-Martin compactification arises by taking $\cF$ to be the set of functions $y\mapsto K(x,y)$, $x\in\bF$.
We use the same symbol for the extended functions
and denote boundary elements by lower case Greek letters. With this construction, \eqref{eq:aslimit} 
and~\eqref{eq:tailgen} are satisfied. In addition, we have the following remarkable properties:
First, all non-negative harmonic functions $h:\bF\to\bR$ can be written as mixtures of
the functions $K(\cdot,\alpha)$, $\alpha\in\partial\bF$. To be precise we recall that $h:\bF\to \bR$
is harmonic if $h(x)=\sum_{y\in \bF} p(x,y)h(y)$ for all $x\in \bF$. Then for each such $h$ with 
$h\ge 0$ and $h(e)=1$ there is a probability measure $\mu_h$ on (the Borel subsets of) the 
boundary $\partial\bF$ such that
\begin{equation}\label{eq:reprharm}
  h(x) = \int K(x,\alpha)\, \mu_h(d\alpha)\ \  \text{ for all } x\in \bF.
\end{equation}
The distribution $\mu_1$ of the limit $X_\infty$ represents the (trivial) harmonic function $h \equiv 1$. 
Secondly, conditioned on a limit value $X_\infty=\alpha$, the process is again a Markov chain, with transition 
probabilities $p^h$ given by
\begin{equation}\label{eq:htransf}
  p^h(x,y) \, =\,  \frac{1}{K(x,\alpha)}\, p(x,y) K(y,\alpha).
\end{equation}
This is an instance of Doob's $h$-transform, with $K(\cdot,\alpha)$ the corresponding harmonic function $h$. Of course,
the interpretation of these transforms by a conditioning on the final value is a natural consequence of the initial idea
of conditioning on the values $y_n$ at time $n$ and then letting $n$ tend to $\infty$.

As it is central to our theme of persisting randomness we briefly explain why (and how)~$X_\infty$ generates the tail
$\sigma$-field, up to null sets, that is, why property~\eqref{eq:tailgen} holds. 

The limit is obviously
$\cT(X)$-measurable, which means that $\sigma(X_\infty)\subset \cT(X)$.
For the other direction we need, for each tail event $A$, a Borel subset $B$ of $\partial\bF$ such that 
\begin{equation}\label{eq:XandT}
  P\bigl(A \bigtriangleup X_\infty^{-1}(B)\bigr) \, =\,  0,
\end{equation}
where $X_\infty^{-1}(B):=\{\omega\in\Omega:\, X_\infty(\omega)\in B\}$. Let $A\in \cT(X)$ and let $1_A$ 
be the associated indicator function;
we may assume that $\kappa:=P(A)>0$. As the state space is graded in the sense that it can be written 
as the disjoint union of the `slices' $\bF_n$ of states that are possible at time $n$, we can define $h:\bF\to [0,\infty)$ by
setting $h(x)=\kappa^{-1}P(A|X_n=x)$ for $x\in\bF_n$, $n\in\bN$. With~\eqref{eq:Markov} it follows that $h$
is harmonic, and it turns out that the measure $\mu_h$ representing $h$ as in~\eqref{eq:reprharm} has a 
density $\Phi$ with respect to the distribution $\mu_1$ of $X_\infty$. The set required in~\eqref{eq:XandT}
can now be given as $B=\Phi^{-1}(\{\kappa^{-1}\})$. 

In particular, if the distribution of $X_\infty$ is concentrated on a single value of the boundary
then $\cT(X)$ is $P$-trivial, so randomness `disappears in the limit'.

\section{Graph limits}%
\label{sec:graph}

Our basis in this section is the recent monograph by~\cite{LovaszBook}, which also gives references to the original
research articles. Let $\bG[n]$ be the set of simple graphs
$G=(V,E)$ with vertex set $V=[n]:=\{1,\ldots,n\}$. The set $\bG[1]$ has only one element, the graph $e=G_1$ with the
single node~1 and no edges.  A number of popular models for randomly growing graphs fits into the framework of
combinatorial Markov chains, with state space $\bF=\bG:=\bigcup_{n=1}^\infty\bG[n]$ and start at 
$G_1$. We work out the boundary for two of them, the uniform attachment process, and the Erd{\H o}s-R\'enyi graphs, 
where we consider two variants of the latter.  We note that the state space compactifications are 
abstract constructions so that the only uniqueness that we may expect is up to homeomorphisms; usually there are
many possibilities for a concrete description.

On its own the question of how to define limits of finite graphs, interpreted as the search for a completion or compactification of
the countable set $\bG$, does not involve any probability and, of course, it can have quite different answers depending
on the specific circumstances. For example, we might distinguish between sparse and dense graphs, referring to
the rate of growth of the number $e(G_n)$ of edges $E(G_n)$ in relation to the number $v(G_n)$ of vertices $V(G_n)$ of
$G_n$ in a sequence $(G_n)_{n\in\bN}\subset \bG$. For the dense case the notion of subgraph sampling has turned out
to be important (there are several equivalent definitions):  For two graphs $G,H\in\bG$ let $t(H,G)$ be the number of 
possibilities to embed $H$ into $G$ or, more formally, with $\Gamma(H,G)$ the set of injective functions 
$\phi:V(H)\to V(G)$, let
\begin{equation}\label{eq:tdef}
     T(H,G):=\bigl\{\phi\in\Gamma(H,G):\, \{\phi(i),\phi(j)\}\in E(G)\Leftrightarrow \{i,j\}\in E(H)\bigr\},
\end{equation}
\begin{equation}\label{eq:tdefrho}
 t(H,G) := \#T(H,G),\quad   \rho(H,G)\; :=\; \frac{(v(G)-v(H))!}{v(G)!}\, t(H,G).
\end{equation}
We then say that a sequence $(G_n)_{n\in\bN}$ converges if for all $H\in\bG$ the relative number $\rho(H,G_n)$ of these
possibilities converges as a sequence of real numbers. The value $\rho(H,G_n)$ can be interpreted as the probability
that, choosing $m=v(H)$ elements of $V(G_n)$ randomly and without replacement, the subgraph of $G_n$ induced on these
nodes is isomorphic to $H$. The convergence may be rephrased in a somewhat abstract manner: We define an embedding of
$\bG$ into the set $[0,1]^\bG$ of functions on $\bG$ with values in the unit interval by
\begin{equation}\label{eq:embedsampling}
  G\, \mapsto\, \bigl(H\mapsto \rho(H,G)\bigr), \quad G\in\bG,
\end{equation}
and then consider the closure of the range of the embedding as a compactification of~$\bG$. Note that
the function space is compact with respect to pointwise convergence by Tychonov's theorem. Viewed this way, 
the similarity to the Doob-Martin compactification becomes apparent, where we use the embedding
\begin{equation}\label{eq:embedMartin}
   x\, \mapsto\, \bigl(y\mapsto K(x,y)\bigr), \quad x\in\bF,
\end{equation}
based on the Martin kernel instead. 
  
Returning to the Markov chain models of randomly growing graphs, we first consider the \emph{uniform attachment} model;
see~\cite[Example~11.39]{LovaszBook}. In order to describe its dynamics suppose that we are in state $G_n\in\bG[n]$ at
time $n$. We then construct $G_{n+1}\in\bG[n+1]$ by adding those edges $\{i,j\}\subset [n+1]$ not (yet) in $G_n$ with
probability $1/(n+1)$, independently of each other. Let $X=(X_n)_{n\in\bN}$ be the corresponding Markov chain, which has
state space $\bG$ and starts at $G_1$. 

Let $\Delta:=\{\{i,j\}\in\bN\times \bN: \, 1\le i < j\}$. For each $\{i,j\}\in\Delta$ we define the edge indicator
$1_{\{i,j\}}:\bG\to \{0,1\}$ to have the value~1 for a graph $G$ if $\{i,j\}$ is an element of the edge set $E(G)$
of~$G$, and~0 otherwise. Recall that $X_n= (V (X_n), E(X_n))$ is a random graph with $V(X_n)=[n]$; let $M_n$ 
be the corresponding (random) 
adjacency matrix. Expressed in graph theoretical terms, part~(a) of the following result shows that the
Doob-Martin convergence associated with the uniform attachment graphs is the same as pointwise convergence 
of the adjacency matrices $M_n$. The limit $M$ may be regarded as the adjacency matrix of the 
limit graph $X_\infty= (V (X_\infty), E(X_\infty))$,  with vertex set $V(X_\infty)=\bN$.

\begin{theorem}\label{thm:uag}
\emph{(a)} The Doob-Martin boundary of the uniform attachment process $X$ consists of  the set $\{0,1\}^\Delta$, where
convergence of a sequence $(G_n)_{n\in\bN}$ of graphs $G_n\in\bG[n]$ to a limit $M\in \{0,1\}^\Delta$ means
that the edge indicators $1_{\{i,j\}}(G_n)$ converge to $M(i,j)$ as $n\to\infty$, for each $\{i,j\}\in\Delta$.

\smallbreak
\emph{(b)} With probability~1, $X_\infty$ is equal to $M\equiv 1$. 
\end{theorem}

\begin{proof}
We compute the Martin kernel. Let $G_m\in\bG[m]$, $G_n\in\bG[n]$, $m< n$, be such that $E(G_m)\subset E(G_n)$. 
For any $\{i,j\}\in  E(G_n)\setminus E(G_m)$ and any $l\in \{m+1,\ldots,n\}$ let $q_{i,j,l}$ be the probability 
that this edge appears in the $X$-sequence from time $l$ onwards. Clearly, if $i<j\le m$, and with the understanding
that an empty product has the value~1,
\begin{align*}
  P\bigl(\{i,j\}\in E(X_n)\big| X_m=G_m\bigr)
        \ &=\ \sum_{l=m+1}^n q_{i,j,l} \\
           &=\ \sum_{l=m+1}^n \biggl(\prod_{k=m+1}^{l-1}\Bigl(1-\frac{1}{k}\Bigr)\biggr) \frac{1}{l} \\
           &=\ \sum_{l=m+1}^n \frac{m}{l(l-1)} \quad
            =\quad 1-\frac{m}{n}.
\end{align*}
If $i<j$ and $j>m$,
\begin{equation*}
   P\bigl(\{i,j\}\in E(X_n)\big| X_m=G_m\bigr)\; =\; \sum_{l=j}^n q_{i,j,l} \; =\; 1-\frac{j-1}{n}.
\end{equation*}
Let $E_j(G):=\{1\le i <j:\, \{i,j\}\in E(G)\}$ and $e_j(G):=\# E_j(G)$. In order to go from $G_m$ 
to $G_n$ the edges in $E_j(G_n)\setminus E_j(G_m)$ with $j=2,\ldots,m$  and those in $E_j(G_n)$ 
with $j=m+1,\ldots,n$ have to enter the 
graph at some time $l\in\{m+1,\ldots,n\}$, and the edges $\{i,j\}$, $1\le i <j$, 
not in $E_j(G_n)$, $j=2,\ldots,n$, must remain unchosen. Hence,  by independence, 
\begin{align*}
   P(X_n=G_n | X_m=G_m)\
       &= \ \prod_{j=2}^m \Bigl(1-\frac{m}{n}\Bigr)^{e_j(G_n)-e_j(G_m)}\Bigl(\frac{m}{n}\Bigr)^{j-1-e_j(G_n)}\\
       &\hspace{1.5cm}    \cdot \prod_{j=m+1}^n \Bigl(1-\frac{j-1}{n}\Bigr)^{e_j(G_n)}\Bigl(\frac{j-1}{n}\Bigr)^{j-1-e_j(G_n)}.
\end{align*}
Using this with $m=1$ we get
\begin{equation*}
  P(X_n=G_n)\; =\; \prod_{j=2}^n \Bigl(1-\frac{j-1}{n}\Bigr)^{e_j(G_n)}\Bigl(\frac{j-1}{n}\Bigr)^{j-1-e_j(G_n)}.
\end{equation*}
Taking ratios we arrive at
\begin{equation*}
  K(G_m,G_n) \, = \,   K_1(G_m,G_n) \cdot  K_2(G_m,G_n)  \cdot  K_3(G_m,G_n) 
\end{equation*}
with
\begin{equation*}
  K_1(G_m,G_n) := \prod_{j=2}^m \Bigl(1-\frac{m}{n}\Bigr)^{-e_j(G_m)} ,\quad
  K_2(G_m,G_n) := \prod_{j=2}^m \Bigl(1-\frac{m+1-j}{n+1-j}\Bigr)^{e_j(G_n)} ,
\end{equation*}
 and
\begin{equation*}
   K_3(G_m,G_n) := \prod_{j=2}^m \Bigl(\frac{j-1}{m}\Bigr)^{e_j(G_n)+1-j}  .
\end{equation*}
Because of $e_j(G)\le j-1$ for all $G\in\bG$, the first two of these factors will converge as 
$n\to\infty$ with $m$ fixed, and the respective limits will always be~1.  

Now let $F_{i,m}$, $1\le i < m$, be the graph with node set $[m]$ 
and a single edge $\{i,m\}$. If this edge appears in $G_n$, then
\begin{equation*}
  K_3(F_{i,m},G_n)\, = \, \prod_{j=2}^m \Bigl(\frac{j-1}{m}\Bigr)^{e_j(G_n)+1-j} 
                        \ \ge \ 1
\end{equation*}
in view of $e_j(G_n)\le j-1$.  As $K(F_{i,m},G_n)=0$ if $\{i,m\}\notin E(G_n)$ this means that Doob-Martin convergence
implies that the limits
\begin{equation}\label{eq:uagcond2}
  M(i,m) := \lim_{n\to\infty} 1_{\{i,m\}}(G_n)\ \in \ \{0,1\}
\end{equation}
exist for all $\{i,m\}\in\Delta$. On the other hand, from \eqref{eq:uagcond2} we obtain the existence of the limits
\begin{equation}\label{eq:uagcond1}
    \lim_{n\to\infty} e_j(G_n) \ \in \ \{0,\ldots,j-1\}
\end{equation}
for all $j\in\bN$, $j\ge 2$, which in turn implies the convergence of $K_3(G_m,G_n)$. 
Taken together this characterizes Doob-Martin convergence as stated in part (a).
In~\eqref{eq:uagcond2} and~\eqref{eq:uagcond1} convergence means that the sequence elements
do not change from some index $n$ onwards.

For the proof of part (b) let $\tau_{ij} := \inf\bigl\{n\in\bN:\, \{i,j\}\in E(X_n)\bigr\}$ 
be the entry time of the edge $\{i,j\}$; we need to show that $P(\tau_{ij}<\infty)=1$ 
for all $\{i,j\}\in\Delta$. This, however, is an easy consequence of the construction of $X$ as we have, for $n\ge j$,
\begin{equation*}
  P(\tau_{ij}> n)\; = \; \prod_{l=j}^n P(\tau_{ij} >l \,|\, \tau_{ij}\ge l)
                          \; =\; \prod_{l=j}^n \Bigl( 1-\frac{1}{l+1}\Bigr) \; =\; \frac{j}{n+1}. 
\end{equation*}
\end{proof}

As a consequence of part~(b) of the theorem, the tail $\sigma$-field of the uniform attachment process is trivial.
Further, using~\eqref{eq:htransf}, the chain conditioned on some limit value $M\in \{0,1\}^\Delta$ 
can easily be described as follows: We proceed as before, but only edges $\{i,j\}\in\Delta$ with $M(i,j)=1$ 
are allowed to enter.

\bigbreak
\textbf{Remark } (a) An embedding interpretation as in~\eqref{eq:embedsampling} and~\eqref{eq:embedMartin}
of the topology in Theorem~\ref{thm:uag} results if we identify a graph with the values of the edge indicators,
\begin{equation}\label{eq:embeduag}
  G\, \mapsto\, \bigl(\{i,j\}\mapsto 1_{\{i,j\}}(G)\bigr), \quad G\in\bG.
\end{equation}
Here the boundary is the full function space $\{0,1\}^\Delta$, which is usually not the case.  

\smallbreak
(b) The graph sequence generated by the uniform attachment model converges in the sampling topology too, and in fact to
what is arguably a more interesting limit; see~\cite[Proposition~11.40]{LovaszBook}. However, this `more
global' topology does not capture the tail information. To be specific, consider a random variable $\xi$ with values in
$\bN$ and define a random element $\alpha$ of the boundary by $\alpha(i,j)=0$ if $j=\xi$ and $\alpha(i,j)=1$ otherwise,
and let $Y$ be the corresponding $h$-transform. In $Y$ the random node $\xi$ remains isolated forever. Then 
$\cT(Y)=_{\text{\rm\tiny a.s.}} \sigma(\xi)$, which means that some randomness persists. The `local' topology in
Theorem~\ref{thm:uag} detects this, whereas from the global point of view
$Y$ and the original chain $X$ are asymptotically indistinguishable.

\smallbreak
(c) Whereas~\eqref{eq:embeduag} is an embedding in the strict sense of being one-to-one, \eqref{eq:embedsampling}
is not: If $G$ and $G'$ are of the same isomorphism type,  then the functions $\rho(\cdot,G)$ and $\rho(\cdot,G')$
coincide; see also Theorem~5.29 and its proof in \cite{LovaszBook}.
\hfill$\triangleleft$  

\bigbreak

The second model that we consider is perhaps the most famous of all random discrete structures: To obtain the
Erd{\H o}s-R\'enyi-Gilbert graph~\cite[p.8]{LovaszBook} or binomial random graph~\cite[p.2]{buchJLR} $Y_n$
with node set $[n]$ and parameter $\theta\in [0,1]$ we include each of the $\binom{n}{2}$ possible edges with 
probability $\theta$, independently of each other. In contrast to the uniform attachment model discussed above the 
variables $Y_n$ are now
defined for each $n$ separately, with
\begin{equation}\label{eq:distY}
  P(Y_n =G_n)\; =\; \theta^{e(G_n)} (1-\theta)^{\binom{n}{2}-e(G_n)}\quad \text{for all } G_n\in \bG[n],
\end{equation}
but there is a well-known and canonical method to combine these into a Markov chain
$Y=(Y_n)_{n\in\bN}$: In order to move from $Y_n$ to $Y_{n+1}$ we add the node $n+1$ and then, 
independently of each other, each of the edges $\{i,n+1\}$, $1\le i\le n$, with probability $\theta$.  

A moment's thought reveals that, in this model, none of the randomness will ever go away, which is the other extreme as
compared with tail triviality. Indeed, it is possible to reconstruct the complete sequence $G_1,\ldots,G_n$ from its
last element $G_n$, which implies that $\cT(Y)=\sigma(Y)$. Roughly, for such chains with perfect memory `the sequence is
the limit'; see also~\cite[Section~9]{EGW1}. In order to formalize this, let $\Psi_m^n(G_n)$ be the graph in $\bG[m]$
that $G_n\in\bG[n]$ induces on $[m]$, that is, we delete the nodes $m+1,\ldots,n$ and the incident edges. This defines a
family $\{\Psi_m^n:\, 1\le m\le n <\infty\}$ of functions $\Psi_m^n:\bG[n]\to \bG[m]$ that is consistent in the sense
that $\Psi^m_l\circ\Psi^n_m=\Psi^n_l$ whenever $l\le m\le n$.  Let $\bG_{\text{\rm proj}}\subset \bG[1]\times
\bG[2]\times \bG[3]\times \cdots$ be the set of all sequences $(G_n)_{n\in\bN}$ with the properties that $G_n\in\bG[n]$
for all $n\in\bN$, and $\Psi_m^n(G_n)=G_m$ for all $m,n \in\bN$ with $m<n$. This set is known as the
projective (or inverse) limit associated with the sequence $(\bG[n])_{n\in\bN}$ and the family 
$\{\Psi_m^n:\, 1\le m\le n <\infty\}$.  
In the set of sequences, we regard a sequence (of sequences) as convergent if the respective elements at any particular
position $l\in\bN$ `freeze', i.e.~converge in the discrete topology on~$\bG[l]$. With the discrete topology the
individual components are compact in view of $\#\bG[l]<\infty$, so that their (infinite) product is compact. 
The set $\bG_{\text{\rm proj}}$ is closed therein, hence compact too.

There is a slightly different point of view that connects this abstract procedure to the material 
in the next section and that is also useful for the description of probability measures on the projective
limit: The transition graph of a perfect memory chain on $\bG$ is a rooted and locally finite tree,
with root $G_1$ and directed edges $(G_n,G_{n+1})$, $G_n=\Psi^{n+1}_n(G_{n+1})$. The projective limit 
then coincides with the boundary of the ends compactification of the transition tree. For a node $G$ of this tree with $k$ vertices 
let
\begin{equation*}
  A_G := \bigl\{(G_n)_{n\in\bN}\in\bG_{\text{\rm proj}}: \, G_k = G\bigr\}
\end{equation*}
be the set paths through $G$.  It is easy to see that a probability measure $\mu$ on the ends compactification is
completely specified by the values $\mu(A_G)$, $G\in\bG$.

\begin{theorem}
\emph{(a)} The boundary of the Doob-Martin compactification of $\bG$ with respect to $Y$ is given by the 
projective limit $\bG_{\text{\rm proj}}$. 

\smallbreak
\emph{(b)} The distribution $\mu$ of $Y_\infty$ is given by $\mu(A_G) = P(Y_m=G)$ if $G\in\bG[m]$.
\end{theorem}

\begin{proof}
Again, we look at the Martin kernel: If $G_m$ is on the unique path from $G_1$ to $G_n$ then 
$P(Y_n=G_n,Y_m=G_m)=P(Y_n=G_n)$ so that $K(G_m,G_n)=1/P(Y_m=G_m)$, and  $K(G_m,G_n)=0$ 
otherwise. From this part (a) of the theorem follows easily. For (b) we note that a boundary point $(G_n)_{n\in\bN}$
is in $A_G$ if and only if $G_m=G$. 
\end{proof}

In this situation conditioning on a limit value leads to a deterministic motion along the sequence of graphs
that represents the limit. Part (b) and~\eqref{eq:distY} imply that the limit distribution is diffuse. In 
particular, the tail $\sigma$-field is not trivial, as we have already noted before.

The perfect memory property is a consequence of the labelling of the nodes in the order of their appearance and the fact
that in the step from $n$ to $n+1$ only edges incident to $n+1$ are added.  Uniform attachment graphs do not have this
property; for example, if $\{1,2\}\in E(X_3)$ then it is not clear whether this edge has been added at time $2$ or $3$.

We now show that in the perfect memory case random relabelling may lead to a more interesting topology.  
We recall that the group $\bS[n]$ of permutations of $[n]$ acts on $\bG[n]$, meaning that
each $\pi_n\in\bS[n]$ defines a function on and to $\bG[n]$ by mapping $G=([n],E(G))$ to
\begin{equation}\label{eq:pionG}
  \pi_n(G) \, :=\, \bigl([n], \bigl\{\{\pi_n(i),\pi_n(j)\}: \, \{i,j\}\in E(G)\bigr\}\bigr).
\end{equation}
Now let $\Pi_n$, $n\in\bN$, be a sequence of independent random variables, with $\Pi_n$ uniformly distributed on
$\bS[n]$. We define $X=(X_n)_{n\in\bN}$ inductively by $X_1\equiv G_1$, $X_{n+1}=\Pi_{n+1}(\tilde X_{n+1})$, where $\tilde
X_{n+1}$ is constructed from $X_n$ as in the chain~$Y$ above: $V(\tilde X_{n+1})=[n+1]$ and $E(\tilde X_{n+1})$ is 
obtained from $E(X_n)$ by adding each of the edges $\{i,n+1\}$, $i\in[n]$, independently with probability $p$.   
In view of the fact that the transition from $X_n$ to $X_{n+1}$ only involves $X_n$ and quantities that are independent 
of $X_n$ the process $X=(X_n)_{n\in\bN}$ is again a Markov chain and it continues to be adapted to $\bG$. 
Also, the distribution~\eqref{eq:distY} is invariant under $\bS[n]$ as the 
action~\eqref{eq:pionG} does not change the number of edges. This implies that $X_n$ and the
variable $Y_n$ from the perfect memory version without random relabelling have the same distribution
(but they will in general not be equal). 
Below, we will refer to the sequence $X$ as the Erd{\H o}s-R\'enyi chain with parameter $\theta$.  

The following result shows that the Doob-Martin compactification for this model leads to the sampling topology 
mentioned at the beginning of this section; see~\eqref{eq:tdef} and~\eqref{eq:tdefrho}.

\begin{theorem}\label{thm:ER} Let $X$ be the Erd{\H o}s-R\'enyi chain with parameter $\theta$, $0<\theta<1$.

\emph{(a)} A sequence $(G_n)_{n\in\bN}$ with $G_n\in\bG[n]$ for all $n\in\bN$ converges in the Doob-Martin compactification
associated with $X$ if and only if the sequences $\big(\rho(H,G_n)\bigr)_{n\in\bN}$ 
converge for every fixed $H\in\bG$.

\smallbreak
\emph{(b)} With probability~1, $X_\infty$ is equal to the function
\begin{equation*}
  H\, \mapsto \, \theta^{e(H)} (1-\theta)^{\binom{v(H)}{2}-e(H)}, \quad H\in\bG.
\end{equation*}
\end{theorem}

\begin{proof}
(a) Let $G_m\in\bG[m]$, $G_n\in\bG[n]$ with $n> m$ be given.  The random relabellings in steps $m+1,\ldots,n$
move the nodes $1,\ldots,m$ of $G_m$ to different positions in $[n]$. This defines a random injective function
$\Phi:[m]\to [n]$, where all $n!/(n-m)!$ possible values $\phi$ of $\Phi$ are equally likely. For 
$\phi\notin T(G_m,G_n)$, with $T$ as in~\eqref{eq:tdef},  we clearly have $P(X_n=G_n|X_m=G_m,\Phi=\phi)=0$, 
whereas for $\phi\in T(G_m,G_n)$,
\begin{equation}\label{eq:inv}
  P(X_n=G_n|X_m=G_m,\Phi=\phi)\, =\, \theta^{e(G_n)-e(G_m)}(1-\theta)^{\binom{n}{2}-e(G_n)-\binom{m}{2}+e(G_m)}.
\end{equation}
Note that the right hand side does not depend on $\phi$.
Using~\eqref{eq:distY} a decomposition with respect to the value of $\Phi$ now gives 
\begin{align*}
  K(G_m,G_n)\ &=\ \frac{\sum_{\phi\in T(G_m,G_n)} P(X_n=G_n|X_m=G_m,\Phi=\phi) \, P(\Phi=\phi|X_m=G_m)}
                                       {P(X_n=G_n)}\\
                       &=\ \frac{(n-m)!\; t(G_m,G_n)}{n! \, \theta^{e(G_m)} (1-\theta)^{\binom{m}{2}-e(G_m)}}
                        \ = \ \frac{\rho(G_m,G_n)}{P(X_m=G_m)}. 
\end{align*}
Clearly, this converges for fixed $G_m$ as $n\to\infty\,$ if and only if $\rho(G_m,G_n)$ does.

\smallbreak
(b) This follows from the sampling interpretation of $\rho$ and the independence of the edge
indicators.
\end{proof}

Again, part (b) implies that the tail $\sigma$-field of $X$ is trivial.
Further, the Erd{\H o}s-R\'enyi chains with different parameter values are easily seen to be $h$-transforms
of each other. 

Theorem~\ref{thm:ER} identifies the Doob-Martin boundary of the Erd{\H o}s-R\'enyi chain as a subset of the
set of all functions $G_\infty:\bG\to [0,1]$. For a concise description of this subset, by `graphons', we refer 
the reader to~\cite{LovaszBook}.

\section{Search trees}%
\label{sec:BST}

The nodes of the complete binary tree $\bV=\{0,1\}^\star$ are finite 0-1 sequences (or words) $u=(u_1,\ldots,u_k)$; we
write $u0=(u_1,\ldots,u_k,0)$ and $u1=(u_1,\ldots,u_k,1)$ for the left and right child of $u$ respectively and, if
$k=|u|>0$, $\bar u=(u_1,\ldots,u_{k-1})$ for its direct ancestor. The concatenation of $u=(u_1,\ldots,u_k)$ and
$v=(v_1,\ldots,v_l)$ is given by $u+v=(u_1,\ldots,u_k,v_1,\ldots,v_l)$. We further write $\partial\bV=\{0,1\}^\infty$ for the
ends compactification of the tree $\bV$.

By a binary tree we mean a subset
$x\subset \bV$ that is prefix-stable or, equivalently, contains the ancestor of each of its non-root elements. If
$u\notin x$, $\bar u\in x$, then we call $u$ external, and we write $\partial x$ for the set of external nodes of $x$.
The (fringe) subtree of $x$ rooted at $u$ is given by $x(u):=\{v\in\bV:\, u+v\in x\}$. Let $\bB_n$ be the set of binary trees with $n$
nodes; $\bB:=\bigcup_{n=1}^\infty \bB_n$.  Prefix stability implies that any $x\in\bB$ can be regarded as a contiguous
subset of $\bV$ and hence be described by its \emph{boundary function}
\begin{equation*}
    B_x:\partial\bV\to \bN, \quad (u_i)_{i\in\bN}\mapsto \min\bigl\{k\in\bN:\, (u_1,\ldots,u_k)\notin x\bigr\}.      
\end{equation*}
(This seems to be the most natural term, but in view of all the other occurrences of boundaries in the present paper, 
`frontier' may be a sensible alternative.)

Given a sequence $(t_n)_{n\in\bN}$ of pairwise distinct real numbers the binary search tree (BST) algorithm generates a
sequence $(x_n)_{n\in\bN}$ of labelled binary trees as follows: The first value is stored at the root node; given $x_n$ the next value
$t_{n+1}$ is stored at the first empty node found when travelling through $x_n$, moving from $u$ to $u0$ if the new value is
smaller that the label of an occupied node and to $u1$ otherwise. This is one of the standard algorithm for searching and
also arises in the context of sorting; see~\cite{Knuth3}, \cite{Mahmoud1} and~\cite{Drmota09}. Suppose that
the $t_i$'s are realizations of independent random variables $\eta_i$, $i\in\bN$, with the same continuous
distribution. Then the random binary trees $X_n$ obtained for $\eta_1,\ldots,\eta_n$, $n\in\bN$, can be collected into a
Markov chain $X=(X_n)_{n\in\bN}$ with a simple transition structure: $X_1$ is the tree that consists of the root node $\emptyset$
only, and in the transition from $X_n$ to $X_{n+1}$ one of the $n+1$ external nodes of $X_n$ is chosen uniformly at
random and incorporated into the tree. The BST chain has $\bB$ as its state space, and $P(X_n\in \bB_n)=1$ for all
$n\in\bN$. 

The Doob-Martin compactification of $\bB$ with respect to $X$ and the distribution of the limit $X_\infty$ were obtained
by~\cite{EGW1} and can be described as follows: $\partial\bB$ is the set of probability measures $\mu$ on $\partial\bV$.
Convergence to $\mu\in\partial\bB$ of a sequence $(x_n)_{n\in\bN}\subset\bB$ means that $a_n:=\#x_n\to\infty$ and that
the relative number $\#x_n (u)/a_n$ of nodes in the subtree rooted at $u\in\bV$ converges to $\mu(A_u)$ for all
$u\in\bV$, where $A_u$ consists of all infinite 0-1 sequences with prefix $u$. As with the transition tree in the
previous section, the values $\mu(A_u)$, $u\in\bV$, determine $\mu$.  

We have $P(X_\infty(A_u)>0)=1$ for all $u\in\bV$.  The distribution of $X_\infty$ is a probability
measure on $\partial\bB$, hence on the set of probability measures on $\partial\bV$, where the latter is endowed with the
$\sigma$-field generated by the projections $\mu\mapsto \mu(A)$, $A$ a Borel subset of $\partial \bV$.
This distribution has the (characterizing) property that the random variables
\begin{equation}\label{eq:xidef}
  \xi_u := \frac{X_\infty(A_{u0})}{X_\infty(A_{u})}, \quad u\in\bV,
\end{equation}
are independent and uniformly distributed on the unit interval. This in turn implies that $X_\infty(A_u)$ can be written
as the product of independent, identically distributed random variables, a fact that we will use repeatedly below. 

In the present section we apply this to the asymptotics of the random functions $B_{X_n}$, $n\in\bN$. The idea
of describing randomly growing sets by their boundary appears in connection with models now known under
the acronym `IDLA' (internal diffusion limited aggregation). This subject area was initiated by~\cite{DiacFul},
an early important contribution is~\cite{Lawletal}. Both papers deal with integer lattices, but the basic
model has since then been applied to various other infinite discrete background sets, for example to the `comb' 
by~\cite{Huss}. BST chains may be seen as an IDLA variant on the background set $\bV$, where the exploration 
process is a reinforced random walk in the sense that the probabilities of moving from $u$ to $u0$ and $u1$
respectively depend on the number of previous particles that have travelled along the respective edge. 
The BST boundary functions have earlier been investigated under the name of `silhouette' in~\cite{GruePoiss,GrSilh}, 
where they were regarded as functions on the unit interval via 
\begin{equation}\label{eq:defbeta}
  \partial\bV \ni u=(u_k)_{k\in\bN} \; \mapsto \; \beta(u) := \frac{1}{2} \; + \; \sum_{k=1}^\infty
  \frac{2u_k-1}{2^{k+1}}\in [0,1]
\end{equation}
(binary rationals do not matter as $X_\infty$ has no atoms). Figure~\ref{fig:rawSilh} shows the boundary functions 
of $X_n$ for various $n$, with pseudorandom data, where~\eqref{eq:defbeta} has been used to display $B_{X_n}$ as 
a function on $[0,1]$.

\begin{figure}
  \begin{center}
      \includegraphics[width=9cm,height=7cm]{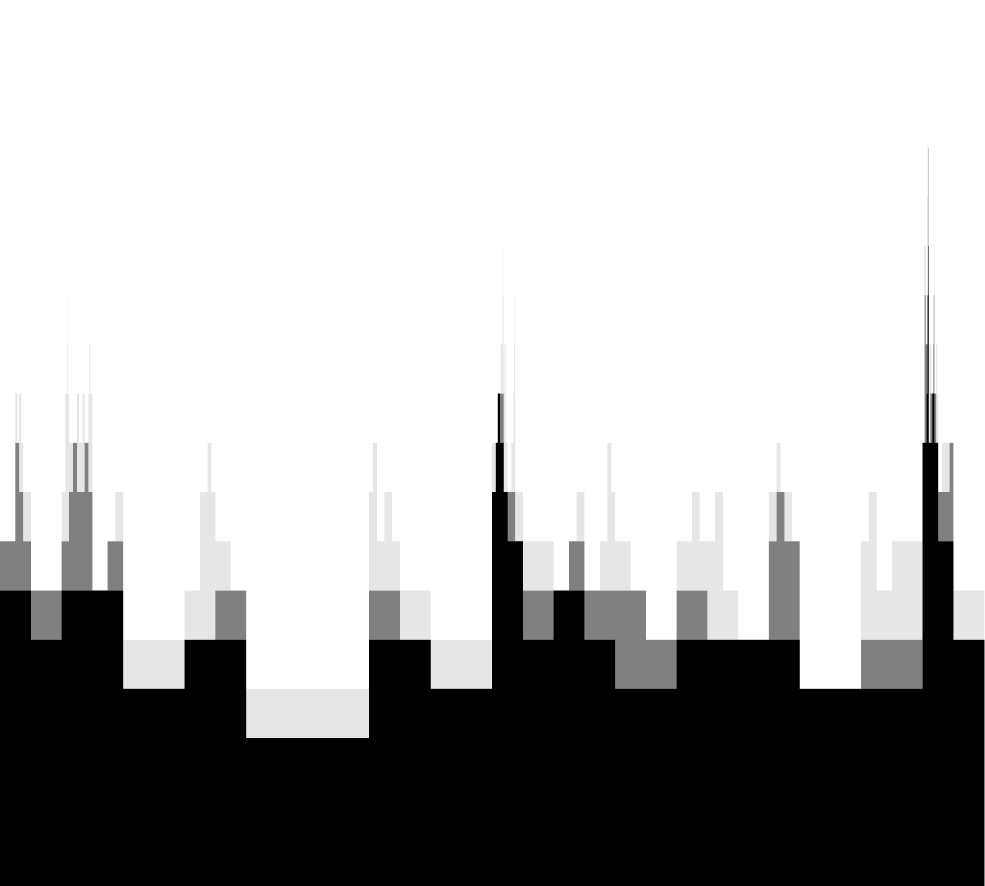}
 \end{center}
\caption{The subgraph of $B_{X_n}$, for $n=50$ (black), $n=100$ (gray) and $n=200$ (light gray).} \label{fig:rawSilh}
\end{figure}

We begin with two real-valued functionals of the boundary functions. First we consider the growth of
the trees along a fixed path through the infinite binary tree. 

\begin{theorem}\label{thm:Bfunc1}
  Let $u\in\partial \bV$ be fixed. Then the tail $\sigma$-field of the sequence $(B_{X_n}(u))_{n\in\bN}$ is $P$-trivial.
\end{theorem} 

A proof can easily be obtained on using the well-known connection to records: The BST dynamics imply that 
$(B_{X_n}(u))_{n\in\bN}$ is identical in distribution to the sequence $(S_n)_{n\in\bN}$, $S_n=\sum_{k=1}^n\zeta_k$,
of partial sums of independent random variables $\zeta_k$ with $P(\zeta_k=1)=1-P(\zeta_k=0)=1/k$, $k\in\bN$,
which also appears when counting records in random samples. It follows that $((n,S_n))_{n\in\bN}$ 
is a Markov chain with state space $\bF=\{(n,i):\, n\in\bN, \,i\in [n]\}$ and  transition probabilities
\begin{equation*}
  p\bigl((n,k),(n+1,k+1)\bigr)\; =\; 1 - p\bigl((n,k),(n+1,k)\bigr) \; = \; \frac{1}{n+1}.
\end{equation*}
The structural similarity to the P\'olya urn mentioned in Section~\ref{sec:intro} should be apparent.
In the records chain, a sequence $((n,k_n))_{n\in\bN}$ of states converges in the sense introduced in 
Section~\ref{sec:summ} if and only if
\begin{equation*}
  \lim_{n\to\infty} \frac{k_n}{\log n} = \alpha\in [0,\infty],
\end{equation*}
which leads to $\bar\bF=\bF\cup [0,\infty]$. This can be proved by `path-counting', the asymptotics of unsigned
Stirling numbers of the first kind, and an interesting monotonicity argument; see~\cite{GnPit} and the references given
there. In the compactification, $S_n$ tends to the constant value~1, which implies triviality of the tail $\sigma$-field
as explained at the end of Section~\ref{sec:summ}.

Let 
\begin{equation*}
 H(0)=0,\quad H(n)=\sum_{k=1}^n1/k\ \text{ for }n\in\bN, 
\end{equation*}
be the harmonic numbers.  From the representation of $B_{X_n}(u)$ as a sum 
of independent Bernoulli variables we obtain the expected value $\,{\E}B_{X_n}(u) = H(n)$
and, using $H(n)\sim\log n$, the distributional
convergence 
\begin{equation}\label{eq:asnorm}
  \frac{1}{\sqrt{\log n}}\bigl(B_{X_n}(u)- \log n\bigr) \ \todistr\; Z\ \text{ as } n\to\infty,
\end{equation}
where $Z$ has a standard normal distribution. However, by Theorem~\ref{thm:Bfunc1}, there is no transformation 
of the random variables $B_{X_n}(u)$, $u\in\partial\bV$ fixed, that leads to strong convergence with a non-degenerate limit.

For the second functional we integrate the boundary functions with respect to the measure
$\lambda$ on $\partial\bV$ given by $\lambda(A_u)=2^{-|u|}$. This is the unique normalized Haar measure if we 
regard the set of infinite 0-1 sequences as a compact group under the pointwise addition modulo 2. 
Recall from~\eqref{eq:xidef} that $\xi_u=X_\infty(A_{u0})/X_\infty(A_u)$, and let 
\begin{equation*}
   C(t)\, =\, 1 \,  +\, \frac{1}{2}\bigl(\log(t)+\log(1-t)\bigr),\quad 0<t<1.
\end{equation*}

\begin{lemma}\label{lem:l0}
The random variables $L_{\infty,k}:= \sum_{|u|<k} 2^{-|u|}C(\xi_u)$, $k\in\bN$, converge
almost surely and in $L^2$ as $k\to\infty$.   
\end{lemma}

\begin{proof} A straightforward calculation shows that ${\E}C(\xi_u) = 0$ and $\var(C(\xi_u)) = 1-(\pi^2/12)$.
In particular, using independence of the $\xi_u$'s,
\begin{equation*}
  \var\Bigl(\,\sum_{|u|=k}2^{-|u|} C(\xi_u)\Bigr)\, = \, 2^{-k}\Bigl(1-\frac{\pi^2}{12}\Bigr),
\end{equation*}
which implies that $\ \bigl(\sum_{|u|\le k}2^{-|u|} C(\xi_u), \sigma(\{\xi_u:\, |u|\le k\})\bigr)_{k\in\bN}$
is an $L^2$-bounded martingale, so that the corresponding limit theorem can be used. 
\end{proof}

We write $L_\infty=\sum_{v\in\bV} 2^{-|v|}C(\xi_v)$ for the limit of $L_{\infty,k}$ as $k \to\infty$. In this series we do
not have absolute convergence: For $k\in\bN$ fixed the mean of the random variable $\sum_{|u|=k}2^{-|u|} |C(\xi_u)|$ is a positive
value that does not depend on $k$.

\begin{theorem}\label{thm:Bfunc2} 
Let $\,L_n:=\int B_{X_n}\, d\lambda$. Then $\;  L_n-H(n) \, \to\, L_\infty\,$ with probability~1 as $n\to\infty$.
\end{theorem}

\begin{proof} Let $\cF_n$ be the $\sigma$-field generated by the first $n$ variables $X_1,\ldots,X_n$ of the BST chain. 
We first show that
\begin{equation}\label{eq:Zrewrite}
  Z_n :=\E[L_\infty|\cF_n] \, = \, \sum_{ u\in X_n} 2^{-|u|} \; -\; H(n).
\end{equation}
By~\eqref{eq:xidef}, the family $\{\xi_u:\, u\in\bV\}$ is a function of the limit $X_\infty$ of the BST sequence.
Hence, using the Markov property of the latter, $\E[C(\xi_u)|\cF_n] = \E[C(\xi_u)|X_n]$.  
From~Proposition~2 and Lemma~4 in~\cite{GrMtree} it is known that the distribution of $\xi_u$ given 
$\# X_n(u0)=i$ and $\# X_n(u1)=j$ is the beta distribution with parameters
$i+1$ and $j+1$, and that for a random variable $\zeta$ with this distribution we have  $\E\log(\zeta) =
H(i)-H(i+j+1)$. In particular, $\E[C(\xi_u)|X_n] = 0$ on $\{u\notin X_n\}$. For $u\in X_n$, using
$\#X_n(u)=\#X_n(u0)+\#X_n(u1)+1$ we obtain
\begin{equation*}
  \E[C(\xi_u)|X_n]\; =\;  \; 1+\frac{1}{2}\bigl(H(\#X_n(u0)) +H(\# X_n(u1))\bigr) -H(\#X_n(u)),
\end{equation*}
which may be written as 
\begin{equation*}
  2^{-|u|} \E[C(\xi_u)|\cF_n]\; = \; 2^{-|u|} -\bigl(\psi_n(u)-\psi_n(u0)-\psi_n(u1)\bigr)
\end{equation*}
with $ \psi_n(u) := 2^{-|u|} H\bigl(\# X_n(u)\bigr)$.
A summation by parts, see~\cite[Lemma~5]{GrMtree}, now gives
\begin{equation*}
  \sum_{ u\in X_n} \bigl(\psi_n(u)-\psi_n(u0)-\psi_n(u1)\bigr)
    \; = \; \psi_n(\emptyset)  -  \sum_{u\in \partial X_n} \psi_n(u).
\end{equation*}
Obviously, $\psi_n(u)=0$ for $u\in\partial X_n$, which completes the proof of~\eqref{eq:Zrewrite}. 

The integral defining $L_n$ may be rewritten as follows,
\begin{equation}\label{eq:intrewrite}
   \int B_{X_n}\, d\lambda\  =\  \sum_{k=1}^\infty k\,2^{-k}\#\{u\in\partial X_n:\, |u|=k\}\;  
                        =\; \sum_{u\in\partial X_n} |u| \, 2^{-|u|}\; =\; \sum_{u\in X_n} 2^{-|u|},
\end{equation}
where the last equality can easily be proved by induction.  Combining this with~\eqref{eq:Zrewrite} and the convergence
of $L^2$-bounded martingales we obtain the assertion.
\end{proof}

As a sum of independent and non-degenerate random variables the limit $L_\infty$ is not almost surely
constant; in particular, the tail $\sigma$-field of the $L$-sequence is not $P$-trivial.

A strong limit theorem for $L_n-H(n)$ has already been obtained in~\citep{GrSilh} by proving directly that
$(L_n-H(n),\cF_n)_{n\in\bN}$ is an $L^2$-bounded martingale. Our approach here differs insofar as it replaces the search
for a suitable martingale by projecting the limit on the natural filtration $(\cF_n)_{n\in\bN}$ of the Markov chain, and
it also provides a representation of the limit in terms of $X_\infty$. The representation in turn leads to an
interpretation of the limit as a distance from $X_\infty$ to the Haar measure $\lambda$: Let $\cG_k$ be the $\sigma$-field
on $\partial\bV$ generated by the sets $A_u$ with $|u|=k$. The Kullback-Leibler divergence $\text{\rm KL}(\mu_1,\mu_2)$
of two measures $\mu_1$ and $\mu_2$ on a measure space $(\Omega,\cF)$, with $\cF$ generated by a partition
$A_1,\ldots,A_l$ of $\Omega$, is given by
\begin{equation*}
  \text{\rm KL}(\mu_1,\mu_2) := \sum_{j=1}^l \mu_1(A_j) \log\biggl(\frac{\mu_2(A_j)}{\mu_1(A_j)}\biggr).
\end{equation*}
 We write $\mu|_{\cG}$ for the restriction of the measure $\mu$ to a sub-$\sigma$-field $\cG$
of its domain~$\cF$. 

\begin{theorem} \label{thm:KL}
With probability~1, $\  L_\infty = \lim_{k\to\infty} \text{\rm KL}(X_\infty|_{\cG_k},\lambda|_{\cG_k})$.
\end{theorem}

\begin{proof} If we restrict the sum in the definition of $L_\infty$ to the nodes of depth less than $k$ then
we obtain with $\phi(u):= 2^{-|u|}\log\bigl(X_\infty(A_u)\bigr)$, using~\eqref{eq:xidef} and a summation by parts
as in the proof of Theorem~\ref{thm:Bfunc2},
\begin{align*}
    \sum_{|u|<k} 2^{-|u|} C(\xi_u)\ 
               &=\ \sum_{|u|< k}2^{-|u|}\biggl(1+2^{-1}\Bigl(\log\frac{X_\infty(A_{u0})}{X_\infty(A_u)} +
                                                                                                 \log\frac{X_\infty(A_{u1})}{X_\infty(A_u)}\Bigr)\biggr)\\
               &=\ \sum_{|u|<k} 2^{-|u|} \; + \; \sum_{|u|<k} \bigl(\phi(u0)+\phi(u1)-\phi(u)\bigr)\\
               &=\ k \; + \; \sum_{|u|=k}\phi(u) \; -\; \phi(\emptyset) \\
               &=\ \sum_{|u|=k} k2^{-|u|} + \sum_{|u|=k} 2^{-k}\log\bigl(X_\infty(A_u)\bigr) 
               \quad =\ \text{\rm KL}(X_\infty|_{\cG_k},\lambda|_{\cG_k}). 
\end{align*}
\end{proof}

\smallbreak
It is tempting to think of the limit as the Kullback-Leibler divergence of $X_\infty$ and $\lambda$. Note, however,
that the density $Z_k:\bV\to [0,\infty)$ of  $X_\infty|_{\cG_k}$ with respect to $\lambda|_{\cG_k}$ is given by
\begin{equation*}
  Z_k(u) \, = \, \frac{X_\infty(A_u)}{2^{-|u|}} \, =\, \prod_{j=1}^k (2\tilde \xi_j), 
\end{equation*}
where $\tilde \xi_j$ is equal to either $\xi_v$ or $1-\xi_v$ with the $\xi$-variables as in~\eqref{eq:xidef}, depending
on the value of the $j$th entry of the sequence $u$, and with $v$ the corresponding length $j-1$ prefix of $u$.  From
this representation as a product of independent, identically distributed and non-degenerate random variables with mean~1
it follows that
\begin{equation*}
  P\bigl(\liminf_{k\to\infty}\, Z_k(u)=0\bigr)\, = \, P\bigl(\limsup_{k\to\infty}\, Z_k(u)=\infty\bigr)\, = \, 1. 
\end{equation*}
In particular,  $X_\infty$ and $\lambda$ are mutually singular with probability 1. Clearly, some cancellation occurs in the sum
defining $L_\infty$, due to the fact that $X_\infty$ is a random measure.
xs
\vspace{5mm}

\begin{figure}[h]
  \begin{center}
      \includegraphics[width=5.5cm,height=5.5cm]{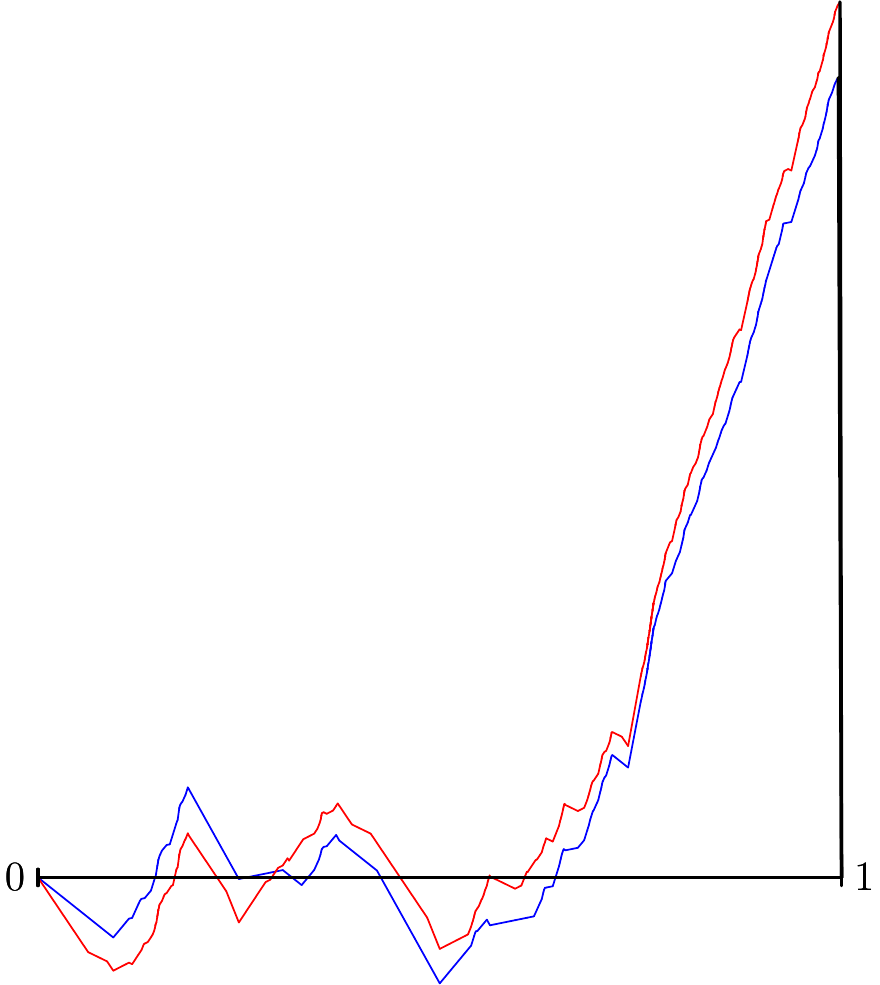} \hspace{1cm}
      \includegraphics[width=5.5cm,height=5.5cm]{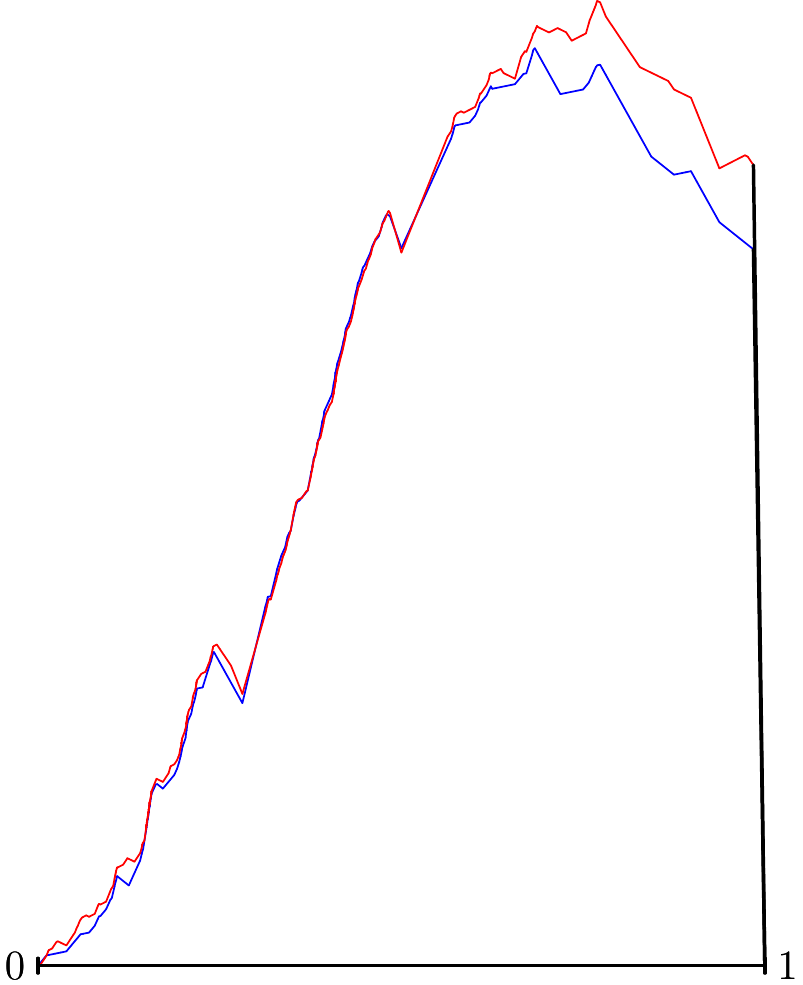}
 \end{center}
\caption{Two values of $Y_n$, with $n=500$ (blue) and $n=1000$ (red).}\label{fig:silh}
\end{figure}

We return to the boundary functions. From Theorem~\ref{thm:Bfunc1} it is clear that we cannot expect these to converge
pointwise; see also Figure~\ref{fig:rawSilh}. Further, the asymptotic normality in~\eqref{eq:asnorm} shows that, at a
specific point,  the
functions increase roughly as $\log n$ but that there are fluctuations of the order $\sqrt{\log n}$.  Hence, apart from
shifting, some smoothing is needed, as has already been noticed in~\citep{GrSilh}.  We first adapt the smoothing
procedure introduced in~\citep{GrSilh} to $\partial\bV$ instead of $[0,1]$ as domain of the random functions. For this we
define a total order on $\partial\bV$ by setting $u\prec v$ for $u=(u_k)_{k\in\bN}$, $v=(v_k)_{k\in\bN}$, with
$u\not=v$, if and only if $u_l=0$ and $v_l=1$ in the first position $l$ where the sequences differ. With the IDLA
connection in mind we are now led to normalizing and smoothing $B_{X_n}$ to $Y_n$ given by
\begin{equation*}
   Y_n(u) := \int_{v\prec u} \bigl(B_{X_n}(v)-H(n)\bigr)\, \lambda(dv)
\end{equation*}
(recall that the $n$th harmonic number is the expectation of $B_{X_n}(v)$ for each $v\in\partial\bV$).
It is easy to deduce from~\cite[Theorem~8]{GrSilh} that $Y_n$ converges in distribution to a process with continuous 
paths. However, as the $Y_n$'s are all defined on the same probability
space, it makes sense to ask whether these variables themselves converge.

Figure~\ref{fig:silh} shows the values of $Y_n(\omega)$ for two $\omega$'s, with $n=500$ and $n=1000$ respectively,
where instead of two such $\omega$'s in the left and the right part of the figure two separate streams of numbers were
used that the present author regards as plausible substitutes for truly random numbers: The two sequences were generated
from alternating blocks of ten digits in the decimal expansion of $\pi-3$, so that the left stream begins with
$t_1=0.1415926535$, $t_2=0.2643383279$, $\ldots$, whereas the right stream has $t_1=0.8979323846$, $t_2=0.5028841971$
and so on.  As in Figure~\ref{fig:rawSilh}, $\partial\bV\,$ is mapped to $[0,1]$ by the function $\beta$ defined
in~\eqref{eq:defbeta} in order to be able to draw the functions.

The figure supports the conjecture that the random functions $Y_n$ themselves converge, and that the limit is not a
fixed function. Incidentally, it also demonstrates the influence of the first few values on the output of the BST
algorithm: The long-term proportion $\# X_n((0))/n$ of nodes in the left subtree is equal to the value of the first input
variable, for example, and for the above $\pi$-data the $t_1$-values are quite different.

The theorem below confirms this conjecture. The theorem also provides a representation of the limit process
in terms of the Doob-Martin limit $X_\infty$ of the BST sequence and, in fact, its proof is closely connected to this
representation. 

We need to specify what convergence of the random functions $Y_n$ means. For this, we define a metric $d$ on $\partial\bV$ by
$d(u,v) = 2^{-l+1}$ for $u,v\in\partial\bV$, $u\not=v$, where $l$ is the first coordinate in which the two sequences
differ as in the definition of the total order on $\partial\bV$; also, $l-1=|u\wedge v|$ where $u\wedge v$ denotes the
longest common prefix (last common ancestor) of $u$ and $v$. This turns $\partial\bV$ into a compact metric space; we
write $C(\partial\bV)$ for the set of continuous functions on $(\partial\bV,d)$. Endowed with the supremum norm,
$\|f\|_\infty=\sup_{u\in\bV} |f(u)|$, $C(\partial\bV)$ becomes a Banach space.  

Further, for $u=(u_k)_{k\in\bN}\in\partial\bV$ let $K(u):=\{k\in\bN:\, u_k=1\}$, and $u(k):=(u_1,\ldots,u_{k-1},0)$
for $k\in K(u)$.  Generalizing the notation introduced above in connection with the second functional of the boundary
functions we write
\begin{equation}\label{eq:defLinfty}
  L_\infty(u):= \sum_{v\ge u} 2^{-|v|} C(\xi_v), \ \ u\in\bV,
\end{equation}
where `$\ge$' now refers to prefix order.
Clearly, $L_\infty(u)$ has the same distribution as $2^{-|u|}L_\infty(\emptyset)$, and we know from Lemma~\ref{lem:l0} that
$L_\infty(\emptyset)=L_\infty$ has zero mean and finite variance.

We require two auxiliary results. 

\begin{lemma}\label{lem:l1} 
\emph{(a)} For each $u\in\partial\bV$, the random variables 
\begin{equation*}
  Y'_{\infty,m}(u):= \sum_{k\in K(u),k\le m} L_\infty(u(k)), \ \  m\in\bN,
\end{equation*}
converge almost surely and in $L^2$ as $m\to\infty$.   

\smallbreak
\emph{(b)} For each $u\in\partial\bV$, the random variables 
\begin{equation*}
  Y''_{\infty,m}(u):= \sum_{k\in K(u),k\le m} 2^{-k} \log\bigl(2^kX_\infty(A_{u(k)})\bigr), \ \ m\in\bN,
\end{equation*}
converge almost surely and in $L^2$ as $m\to\infty$. 
\end{lemma}

\begin{proof} (a) For all $m,n\in\bN$ we have
\begin{equation*}
   \bigl\|Y'_{\infty,m}(u)-Y'_{\infty,n}(u)\bigr\|_2\; 
     =\; \E\,\Bigl\|\sum_{\substack{k\in K(u),\\ m\wedge n < k\le m\vee n}} L_\infty\bigl(u(k)\bigr)\Bigr\|_2
     \; \le \; 2^{-(m\wedge n-1)/2} \, \|L_\infty\|_2.
\end{equation*}
This shows that that $(Y'_{\infty,m}(u))_{m\in\bN}$ is a Cauchy sequence in $L^2$ and that, with $Y'_\infty$ the limit,
\begin{equation*}
   \E\bigl(Y'_{\infty,m} -Y'_\infty\bigr)^2 = O(2^{-m}).
\end{equation*}
In particular, using Chebyshev's inequality, we get
\begin{equation*}
  \sum_{m=1}^\infty P(|Y'_{\infty,m} -Y'_\infty|>\epsilon) < \infty\ \text{ for all } \epsilon>0,
\end{equation*}
which implies almost sure convergence. 

(b) It follows from~\eqref{eq:xidef} that 
\begin{equation*}
  \log X_\infty(A_u)  = \sum_{j=1}^{|u|} \log \tilde\xi_j,
\end{equation*}
with $\tilde \xi_1, \ldots,\tilde \xi_{|u|}$ independent and uniformly distributed on the unit interval; see also the 
discussion following Theorem~\ref{thm:KL}. This
leads to
\begin{equation*}
  \bigl\| \log\bigl(2^{-|u|} X_\infty(A_u))\bigr\|_2 = O(|u|).
\end{equation*}
Using this bound on the $L^2$-norm of the individual summands we can now proceed as in 
the proof of part~(a).
\end{proof}

In view of Lemma~\ref{lem:l1} it makes sense to define two random functions
$Y_\infty'=(Y'_\infty(u))_{u\in\partial\bV}$ and  $Y_\infty''=(Y''_\infty(u))_{u\in\partial\bV}$ 
on $\partial\bV$ by
\begin{equation}\label{eq:defYY}
  Y'_\infty(u) = \sum_{k\in K(u)} L_\infty(u(k)), \quad Y''_\infty(u)= \sum_{k\in K(u)} 2^{-k} \log\bigl(2^kX_\infty(A_{u(k)})\bigr).
\end{equation}
The random functions $Y_\infty'$ and $Y_\infty''$
may also be regarded as stochastic processes with time parameter $u\in\partial \bV$.

\begin{lemma}\label{lem:l2}
\emph{(a)} With probability~1, the processes $Y'_\infty$ and $Y''_\infty$ have continuous paths.  

\smallbreak
\emph{(b)} Both processes are integrable in the sense that 
\begin{equation*}
  \E\|Y'_\infty\|_\infty<\infty\ \text{ and }\ \E\|Y''_\infty\|_\infty<\infty.
\end{equation*}
\end{lemma}

\begin{proof} We consider $Y'_\infty$ first. 
For the proof of continuity we adapt the well-known chaining argument, see e.g.~\cite [p.35]{Kall}, to the present situation. 
Let 
\begin{equation*}
  \rho_k := \max\bigl\{|L_\infty(v)|:\, v\in\bV, \, |v|=k\bigr\}, \ \ k\in\bN.
\end{equation*}
For nodes $v$ on a fixed level $k$ the variables $L_\infty(v)$  are independent. Using 
$\rho_k^2\le \sum_{|v|=k}L_\infty(v)^2$ we obtain
\begin{equation*}
  \E\rho_k^2 \le \sum_{|v|=k} {\E}L_\infty(v)^2
         = \sum_{|v|=k} \var(L_\infty(v)) = 2^{-k} \var(L_\infty),
\end{equation*}
from which it follows that
\begin{equation}\label{eq:chainbound}
  \E\Bigl(\sum_{k=1}^\infty 2^{k/2}\rho_k^2\Bigr) <  \infty.
\end{equation}
This implies that on a set $A$ of probability~1 we have
\begin{equation*}
  \rho_k(\omega) \le C(\omega) \, 2^{-k/4} \ \text{ for all } k\in\bN,\, \omega\in A,
\end{equation*}
with some $C(\omega)<\infty$ that does not depend on $k$. Suppose now that $u,v\in\bV$ are such that $d(u,v)\le
\epsilon$. Then the first $l=l(\epsilon)=\lceil -\log_2(\epsilon)\rceil$ entries of $u$ and $v$ coincide, so that
their connecting path does not go below height~$l$. With Lemma~\ref{lem:l1} and the triangle inequality we therefore get
\begin{equation*}
  \bigl| Y'_\infty(u)(\omega) - Y'_\infty(v)(\omega)\bigr| \; \le\; 2\sum_{k=l+1}^\infty \rho_k(\omega)
               \; \le \; 13\, C(\omega) \, \epsilon^{1/4}
\end{equation*}
whenever $\omega\in A$. This implies that almost all paths of $Y'_\infty$ are continuous.

For the proof of integrability we first note that $\|Y'_\infty\|_\infty\le \sum_{k=1}^\infty
\rho_k$. Using~\eqref{eq:chainbound} and
\begin{equation*}
  \Bigl(\sum_{k=1}^\infty |a_k|\, |b_k|\Bigr)^2\, \le \, \Bigl(\sum_{k=1}^\infty a_k^2\Bigr) \,\Bigl(\sum_{k=1}^\infty b_k^2\Bigr) 
\end{equation*}
with $a_k=2^{-k/4}$ we see that the upper bound has finite mean. 

As in the proof of the previous lemma, the arguments used for $Y_\infty'$ can 
be transferred to the other process $Y_\infty''$: We now put
\begin{equation*}
  \sigma_k := \max\bigl\{2^{-k} \bigl|\log(2^k X_\infty(A_v))\bigr|:\, v\in\bV, \, |v|=k\bigr\}, \ \ k\in\bN,
\end{equation*}
and again, we will show that these decrease rapidly enough as $k\to\infty$. However, we no longer have independence 
of the individual random variables in the maximum, so we need a different argument. As in~\citep{GrMtree} in connection with
the maximum of the probabilities $X_\infty(A_u)$, $|u|=k$, we use the connection to the branching random walks discussed 
by~\cite{Biggins}.   This rests upon the observation that the variables
\begin{equation*}
  Y_u := -\log X_\infty(A_u), \quad |u|=k,
\end{equation*}
are the positions of the members of the $k$th generation in a branching random walk
with offspring distribution $\delta_2$, meaning that each particle has exactly two descendants, and with
\begin{equation*}
   Z:= \delta_{-\log\xi} +\delta_{-\log(1-\xi)},\quad \cL(\xi)=\unif(0,1),
\end{equation*}
for the point process of the positions of the children relative to their parent. Let
\begin{equation}\label{eq:mfunc}
   m(\theta) := \E\Bigl(\int e^{\theta t}\, Z(dt)\Bigr)\, =\, \frac{2}{1-\theta}, \quad \theta<1,
\end{equation}
and let $Z_+^{(k)}(t)$ be the number of particles in generation $k$ that are located to the right of $t$. 
The random measure $Z^{(k)}$ is the $k$th convolution power of $Z$, which leads to
\begin{equation*}
  \E\Bigl(\,\int e^{\theta t} Z^{(k)}(dt)\Bigr)\, = \, m(\theta)^k,
\end{equation*}
so that 
\begin{align*}
  P\bigl(\max \{Y_{u}:\, |u|=k\} \ge  t\bigr)\ &\le \  {\E}Z_+^{(k)}(t) \\ 
               &\le\ e^{-\theta t}\, \E\Bigl(\,\int_{[t,\infty)} e^{\theta s}Z^{(k)}(ds)\Bigr)\\
               &\le \; e^{-\theta t} \, m(\theta)^k 
\end{align*} 
for all  $t\in\bR$, $0\le \theta < 1$ and $k\in\bN$.  Using~\eqref{eq:mfunc} we get,
with $\theta=1/2$,
\begin{align*}
    \E\bigl(\max \{Y_{u}:\, |u|=k\}\bigr)
              \ &= \ \int_0^\infty P\bigl(\max \{Y_{u}:\, |u|=k\}\ge t \bigr)\, dt\\
                 &\le\ \Bigl(\frac{5}{4}\Bigr)^k \; +\; \int_{(5/4)^k}^\infty e^{-t/2} \,4^k\, dt\\
                 &=\ O\biggl(\Bigl(\frac{4}{3}\Bigr)^k\biggr),
\end{align*}
so that, for some constant $C<\infty$,
\begin{equation*}
  \E\biggl(\sum_{k=1}^\infty \Bigl(\frac{5}{4}\Bigr)^k\sigma_k\biggr)
         \ \le \ C\, \sum_{k=1}^\infty \Bigl(\frac{5}{4}\Bigr)^k 2^{-k} \, \E\bigl(\max \{Y_{u}:\, |u|=k\}\bigr)
        \ <\ \infty.
\end{equation*}
Using this instead of~\eqref{eq:chainbound} we can now proceed as in the first part of the proof.
\end{proof}

There is obviously room to spare in the above chaining inequalities; tightening these leads to path
properties beyond continuity.

In the proof of our final result we will use infinite-dimensional martingales; see~\cite[Chapter~V-2]{NeveuMart}.
For this, we require separability of the Banach space $\bigl(C(\partial\bV),\|\cdot\|_\infty\bigr)$: The sets $A_u$,
$u\in\bV$, are closed and open in $\partial\bV$, so their indicator functions are continuous. Moreover, the
intersection of two such sets is again of this form, and the indicator functions separate the points of $\partial\bV$.
The required separability now follows on using the Stone-Weierstra{\ss} theorem. 

\begin{theorem}\label{thm:isil}
With probability~1, $Y_n$ converges in $\bigl(C(\partial\bV),\|\cdot\|_\infty\bigr)$ to $Y_\infty:=Y'_\infty+ Y''_\infty\,$ as $n\to\infty$.
\end{theorem}

\begin{proof} For all $u\in\bV$, $n\in\bN$,
 \begin{align*}
  \int_{A_u} B_{X_n}\, d\lambda \ &=\ \sum_{v\in\partial X_n, v\ge u} |v| \, 2^{-|v|} \\
                   &=\ \sum_{w\in \partial X_n(u)} (|u|+|w|)\, 2^{-|u|-|w|}\\
                   &=\ 2^{-|u|}\Bigl(|u|\sum_{w\in\partial X_n(u)}2^{-|w|} + \sum_{w\in\partial X_n(u)}|w|\,2^{-|w|}\Bigr)\\
                   &=\ 2^{-|u|}\Bigl(|u| \; + \sum_{w\in X_n(u)}2^{-|w|}\Bigr) \\
                   &=\ |u|\, 2^{-|u|} \; + \sum_{v\in X_n, v\ge u} 2^{-|v|}.
\end{align*}
Further, with $L_\infty(u)$ as in~\eqref{eq:defLinfty}, 
\begin{align*}
   \E[L_\infty(u)|\cF_n] \ &=\ \E\Bigl[\sum_{v\in\bV,v\ge u} 2^{-|v|} C(\xi_v)\Big| \cF_n\Bigr]\\
                                     &=\ \sum_{v\in X_n, v\ge u} \Bigl(2^{-|v|} \, - \, \bigl(\psi_n(u)-\psi_n(u0)-\psi_n(u1)\bigr)\Bigr)\\
                                     &=\ \sum_{v\in X_n(u),v\ge u} 2^{-|v|} \; -\;  2^{-|u|} H \bigl(\#X_n(u)\bigr), 
\end{align*}
where we have used the same notation and the same arguments as in the proof of~\eqref{eq:Zrewrite}. Combining these we
get, for all $u\in\bV$,
\begin{equation}\label{eq:decomp0}
  \int_{A_u} \bigl(B_{X_n} - H(n)\bigr)\, d\lambda \; 
           = \; \E [L_\infty(u)|\cF_n] \, + \, 2^{-|u|}\bigl( H(\#X_n(u))-H(n) + |u|\bigr).
\end{equation}
Now let $u=(u_k)_{k\in\bN}\in\partial\bV\,$ be such that $\# K(u)<\infty$. The integration range appearing in the
definition of $Y_n(u)$ can be decomposed as follows, 
\begin{equation*}
  I_u:=\{v\in\partial\bV:\, v\prec u\} = \sum_{k\in K(u)} A_{u(k)}.
\end{equation*}
With~\eqref{eq:decomp0} we thus obtain, setting $Y'_n := \E\bigl[Y'_\infty\big|\cF_n\bigr]$,
\begin{align*}
        Y_n(u) \ &=\ \sum_{k\in K(u)} \int_{A_u} \bigl(B_{X_n} - H(n)\bigr)d\lambda \\
              &=\ Y'_n(u) \, +\, \sum_{k\in K(u)}2^{-k}\bigl( H(\#X_n(u(k)))-H(n)+k\bigr).
\end{align*}
From Lemma~\ref{lem:l2} we know that $Y'_\infty$ is a
$C(\partial\bV)$-valued integrable random variable. Hence $(Y'_n,\cF_n)_{n\in\bN}$ is a martingale with values in the
separable Banach space $\bigl(C(\partial\bV),\|\cdot\|_\infty\bigr)$, and by~\cite[Proposition V.2.6]{NeveuMart} $Y'_n$
converges almost surely in this space to $Y'_\infty$ as $n\to\infty$.

It remains to prove that, as $n\to\infty$,
\begin{equation}\label{eq:Hconv}
  \sum_{k\in K(u)}2^{-k}\bigl( H(\#X_n(u))-H(n)+k\bigr)\; \to \; Y_\infty''(u),
\end{equation}
with probability~1 and, with both sides regarded as functions of $u$, in $\bigl(C(\partial\bV),\|\cdot\|_\infty\bigr)$. 
For this, we first show that 
\begin{equation}\label{eq:condexplog}
  \E\bigl[\log\bigl(X_\infty(A_u)\big|\cF_n\bigr]\; =\, H\bigl(\#X_n(u)\bigr) - H(n)\quad \text{for all } u\in\bV,\, n\in\bN.
\end{equation}
Clearly, for $u=\emptyset$, both sides of~\eqref{eq:condexplog} are equal to 0. 
If $u=(u_1,\ldots,u_k,u_{k+1})$ with $u_{k+1}=0$ then, from~\eqref{eq:xidef},
\begin{equation*}
  \log X_\infty(A_u) \, = \; \log X_\infty(A_{\bar u}) + \log\xi_{\bar u}.
\end{equation*}
Further, we know from the proof of Theorem~\ref{thm:Bfunc2} that
\begin{equation*}
  \E[\log \xi_{\bar u}|\cF_n]\, = \, H\bigl(\#X_n(\bar u0)\bigr)-H\bigl(\#X_n(\bar u)\bigr).
\end{equation*}
Hence, if~\eqref{eq:condexplog} holds for $\bar u$, then so it does for $u$. The same arguments work in the case 
$u_{k+1}=1$, with $1-\xi_{\bar u}$ and $\#X_n(\bar u1)$ instead of $\xi_{\bar u}$ and $\#X_n(\bar u0)$ respectively.
This completes the induction proof for~\eqref{eq:condexplog}.

As in the first part of the proof we now get
\begin{equation*}
  \E\bigl[Y''_\infty(u)\big|\cF_n\bigr]\, 
        = \, \sum_{k\in K(u)} 2^{-k}\Bigl(H\bigl(\#X_n(u(k))\bigr)-H(n)+k\Bigr).
\end{equation*}
From this \eqref{eq:Hconv} follows on using the infinite-dimensional martingale convergence theorem again.
\end{proof}

\section{Comments and complements}
\label{sec:compl}
We collect some references to related work and also put the above results into a larger perspective.

\smallbreak
{\bf (a)} The approach of the present paper is not limited to graphs and search trees but may be
used quite generally in the context of combinatorial Markov chains. For an elementary
introduction to such processes and their boundaries, with many examples and algorithms, see~\cite{GrKMK}
(written in simple German).

\smallbreak
{\bf (b)} In concrete cases, the results provided by a general method such as the Doob-Martin approach 
can often be obtained more directly, using the additional structures then present. For example,   in~\cite{GrMtree} 
a proof of the basic BST result from~\cite{EGW1} is given that is based on the BST algorithm; this direct approach also 
leads to a representation of $X_\infty$ in terms of the input sequence. Obviously, the same applies to the theory
of graph limits, but the exposition of a common structure provided by a general theory may lead to a deeper 
understanding of such individual cases. 

\smallbreak
{\bf (c)} As seen above, the boundary theory approach may lead to strong
limit theorems for discrete structures and their functionals, occasionally improving on previous results.
In~\cite{GrMtree} such an amplification from convergence in distribution to convergence of the random variables is
carried out for the Wiener index of search trees, where distributional convergence had earlier been obtained
by~\cite{Nein02} with the contraction method. In both cases it is instructive to compare the proofs, which are quite
different and seem to be less involved for the stronger result (once the Doob-Martin compactification has been 
worked out). The records chain provides an example where we have distributional convergence 
with a non-degenerate limit, but where a strong limit is necessarily degenerate, i.e.~constant. 

\smallbreak
{\bf (d)} At a qualitative level functionals of discrete random structures may have a non-trivial tail $\sigma$-field,
which we interpret as persisting randomness, or they may not, even if the structures themselves show such a
persistence; see Theorem~\ref{thm:Bfunc1}.  A similar phenomenon has been observed in connection with the 
subtree size profile of binary search trees by~\cite{DeGr3}. 

\bigbreak
{\bf Acknowledgements}. I thank Steve Evans, Klaas Hagemann, Anton Wakolbinger and Wolfgang Woess for helpful
discussions. Further, I am grateful to the referee for comments that have led to numerous improvements. A talk based 
on the material of this paper was given at the 2014 AofA conference in Paris; I also thank the
participants for their feedback.

\bibliographystyle{abbrvnat}
\bibliography{persistRDS-DMTCS-FINAL}

\begin{thebibliography}{27}
\providecommand{\natexlab}[1]{#1}
\providecommand{\url}[1]{\texttt{#1}}
\expandafter\ifx\csname urlstyle\endcsname\relax
  \providecommand{\doi}[1]{doi: #1}\else
  \providecommand{\doi}{doi: \begingroup \urlstyle{rm}\Url}\fi

\bibitem[Biggins(1977)]{Biggins}
J.~D. Biggins.
\newblock Chernoff's theorem in the branching random walk.
\newblock \emph{J. Appl. Probability}, 14\penalty0 (3):\penalty0 630--636,
  1977.

\bibitem[Blackwell and Kendall(1964)]{BK1964}
D.~Blackwell and D.~Kendall.
\newblock The {M}artin boundary of {P}\'olya's urn scheme, and an application
  to stochastic population growth.
\newblock \emph{J. Appl. Probability}, 1:\penalty0 284--296, 1964.

\bibitem[Dennert and Gr{\"u}bel(2010)]{DeGr3}
F.~Dennert and R.~Gr{\"u}bel.
\newblock On the subtree size profile of binary search trees.
\newblock \emph{Combin. Probab. Comput.}, 19\penalty0 (4):\penalty0 561--578,
  2010.

\bibitem[Diaconis and Fulton(1991)]{DiacFul}
P.~Diaconis and W.~Fulton.
\newblock A growth model, a game, an algebra, {L}agrange inversion, and
  characteristic classes.
\newblock \emph{Rend. Sem. Mat. Univ. Politec. Torino}, 49\penalty0
  (1):\penalty0 95--119 (1993), 1991.
\newblock Commutative algebra and algebraic geometry, II (Italian) (Turin,
  1990).

\bibitem[Doob(1959)]{Doob59}
J.~L. Doob.
\newblock Discrete potential theory and boundaries.
\newblock \emph{J. Math. Mech.}, 8:\penalty0 433--458; erratum 993, 1959.

\bibitem[Drmota(2009)]{Drmota09}
M.~Drmota.
\newblock \emph{Random trees. {A}n interplay between combinatorics and
  probability}.
\newblock Springer, Wien, 2009.

\bibitem[Evans et~al.(2012)Evans, Gr{\"u}bel, and Wakolbinger]{EGW1}
S.~N. Evans, R.~Gr{\"u}bel, and A.~Wakolbinger.
\newblock Trickle-down processes and their boundaries.
\newblock \emph{Electron. J. Probab.}, 17:\penalty0 no. 1, 58, 2012.

\bibitem[{Evans} et~al.(2014){Evans}, {Gr{\"u}bel}, and {Wakolbinger}]{EGW2}
S.~N. {Evans}, R.~{Gr{\"u}bel}, and A.~{Wakolbinger}.
\newblock {Doob--Martin boundary of R\'emy's tree growth chain}.
\newblock \emph{ArXiv e-prints}, Nov. 2014.

\bibitem[Gnedin and Pitman(2005)]{GnPit}
A.~Gnedin and J.~Pitman.
\newblock Exchangeable {G}ibbs partitions and {S}tirling triangles.
\newblock \emph{Zap. Nauchn. Sem. S.-Peterburg. Otdel. Mat. Inst. Steklov.
  (POMI)}, 325\penalty0 (Teor. Predst. Din. Sist. Komb. i Algoritm. Metody.
  12):\penalty0 83--102, 244--245, 2005.

\bibitem[Gr{\"u}bel(2005)]{GruePoiss}
R.~Gr{\"u}bel.
\newblock A hooray for {P}oisson approximation.
\newblock In \emph{2005 {I}nternational {C}onference on {A}nalysis of
  {A}lgorithms}, Discrete Math. Theor. Comput. Sci. Proc., AD, pages 181--191
  (electronic). Assoc. Discrete Math. Theor. Comput. Sci., Nancy, 2005.

\bibitem[Gr{\"u}bel(2009)]{GrSilh}
R.~Gr{\"u}bel.
\newblock On the silhouette of binary search trees.
\newblock \emph{Ann. Appl. Probab.}, 19\penalty0 (5):\penalty0 1781--1802,
  2009.

\bibitem[Gr{\"u}bel(2013)]{GrKMK}
R.~Gr{\"u}bel.
\newblock Kombinatorische {M}arkov-{K}etten.
\newblock \emph{Math. Semesterber.}, 60\penalty0 (2):\penalty0 185--215, 2013.

\bibitem[Gr{\"u}bel(2014)]{GrMtree}
R.~Gr{\"u}bel.
\newblock Search trees: Metric aspects and strong limit theorems.
\newblock \emph{Ann. Appl. Probab.}, 24:\penalty0 1269--1297, 2014.

\bibitem[Huss and Sava(2012)]{Huss}
W.~Huss and E.~Sava.
\newblock Internal aggregation models on comb lattices.
\newblock \emph{Electron. J. Probab.}, 17:\penalty0 no. 30, 21, 2012.

\bibitem[Janson et~al.(2000)Janson, {\L}uczak, and Rucinski]{buchJLR}
S.~Janson, T.~{\L}uczak, and A.~Rucinski.
\newblock \emph{Random Graphs}.
\newblock Wiley, New York, 2000.

\bibitem[Ka{\u\i}manovich and Vershik(1983)]{KaimVersh}
V.~A. Ka{\u\i}manovich and A.~M. Vershik.
\newblock Random walks on discrete groups: boundary and entropy.
\newblock \emph{Ann. Probab.}, 11\penalty0 (3):\penalty0 457--490, 1983.

\bibitem[Kallenberg(1997)]{Kall}
O.~Kallenberg.
\newblock \emph{Foundations of modern probability}.
\newblock Springer, New York, 1997.

\bibitem[Kelley(1955)]{Kel55}
J.~L. Kelley.
\newblock \emph{General topology}.
\newblock D. Van Nostrand Company, Inc., Toronto-New York-London, 1955.

\bibitem[Knuth(1973)]{Knuth3}
D.~E. Knuth.
\newblock \emph{The art of computer programming. {V}olume 3, {S}orting and
  searching}.
\newblock Addison-Wesley Publishing Co., Reading, Mass.-London-Don Mills, Ont.,
  1973.

\bibitem[Lawler et~al.(1992)Lawler, Bramson, and Griffeath]{Lawletal}
G.~F. Lawler, M.~Bramson, and D.~Griffeath.
\newblock Internal diffusion limited aggregation.
\newblock \emph{Ann. Probab.}, 20\penalty0 (4):\penalty0 2117--2140, 1992.

\bibitem[Lov{\'a}sz(2012)]{LovaszBook}
L.~Lov{\'a}sz.
\newblock \emph{Large networks and graph limits}, volume~60 of \emph{American
  Mathematical Society Colloquium Publications}.
\newblock American Mathematical Society, Providence, RI, 2012.

\bibitem[Mahmoud(1992)]{Mahmoud1}
H.~M. Mahmoud.
\newblock \emph{Evolution of random search trees}.
\newblock John Wiley \& Sons Inc., New York, 1992.

\bibitem[Neininger(2002)]{Nein02}
R.~Neininger.
\newblock The {W}iener index of random trees.
\newblock \emph{Combin. Probab. Comput.}, 11\penalty0 (6):\penalty0 587--597,
  2002.

\bibitem[Neveu(1975)]{NeveuMart}
J.~Neveu.
\newblock \emph{Discrete-parameter martingales}.
\newblock North-Holland, Amsterdam, revised edition, 1975.

\bibitem[Sawyer(1997)]{Saw97}
S.~A. Sawyer.
\newblock Martin boundaries and random walks.
\newblock In \emph{Harmonic functions on trees and buildings ({N}ew {Y}ork,
  1995)}, volume 206 of \emph{Contemp. Math.}, pages 17--44. Amer. Math. Soc.,
  Providence, RI, 1997.

\bibitem[Woess(2000)]{Woess1}
W.~Woess.
\newblock \emph{Random walks on infinite graphs and groups}.
\newblock Cambridge University Press, Cambridge, 2000.

\bibitem[Woess(2009)]{Woess2}
W.~Woess.
\newblock \emph{Denumerable {M}arkov chains. {G}enerating functions, boundary
  theory, random walks on trees}.
\newblock European Mathematical Society (EMS), Z\"urich, 2009.

\end{thebibliography}

\end{document}